\documentclass[reqno, 12pt, reqno]{amsart}

\usepackage{amsfonts, amsthm, amsmath, amssymb}
\usepackage{hyperref}
\hypersetup{colorlinks=false}

\usepackage[margin=1.5in]{geometry}

\RequirePackage{mathrsfs} \let\mathcal\mathscr

\numberwithin{equation}{section}

\newtheorem{theorem}{Theorem}[section]
\newtheorem{lemma}[theorem]{Lemma}

\newtheorem{conjecture}[theorem]{Conjecture}

\theoremstyle{definition}
\newtheorem*{ack}{Acknowledgements}
\newtheorem{remark}[theorem]{Remark}

\renewcommand{\d}{\mathrm{d}}
\renewcommand{\phi}{\varphi}
\renewcommand{\rho}{\varrho}

\newcommand{\0}{\mathbf{0}}

\newcommand{\ZZ}{\mathbb{Z}}

\newcommand{\NN}{\mathbb{N}}

\newcommand{\RR}{\mathbb{R}}

\newcommand{\cG}{\mathcal{G}}

\renewcommand{\leq}{\leqslant}
\renewcommand{\le}{\leqslant}
\renewcommand{\geq}{\geqslant}

\renewcommand{\bar}{\overline}

\newcommand{\x}{\mathbf{x}}

\renewcommand{\c}{\mathbf{c}}
\renewcommand{\v}{\mathbf{v}}
\renewcommand{\u}{\mathbf{u}}
\newcommand{\z}{\mathbf{z}}

\renewcommand{\b}{\mathbf{b}}
\renewcommand{\a}{\mathbf{a}}

\newcommand{\ve}{\varepsilon}
\newcommand{\e}{\mathbf{e}}

\newcommand{\bxi}{\boldsymbol{\xi}}

\DeclareMathOperator{\vol}{vol}
\DeclareMathOperator{\supp}{supp}

\newcommand{\sumstar}{\sideset{}{^*}\sum}

\renewcommand{\phi}{\varphi}

\newcommand{\BR}{{\mathbb{R}}}

\newcommand{\vect}[1]{{\boldsymbol{#1}}}

\renewcommand{\mod}{\mathop{\rm mod}\nolimits}

\begin{document}

\date{\today}

\title[Twisted Linnik implies optimal covering exponent]{Twisted Linnik implies optimal covering exponent for $S^3$}

\author{T.D. Browning}
\author{
 V. Vinay Kumaraswamy}
\author{R.S. Steiner}
\address{School of Mathematics\\
University of Bristol\\ Bristol\\ BS8 1TW
\\ UK}

\email{t.d.browning@bristol.ac.uk}
\email{vinay.visw@gmail.com}
\email{raphael.steiner@bristol.ac.uk}

\begin{abstract}
We  show  that a  twisted variant of Linnik's conjecture on sums of Kloosterman sums leads to an optimal covering exponent  for $S^3$.
\end{abstract}

\maketitle

\thispagestyle{empty}

\setcounter{tocdepth}{1}
\tableofcontents

\section{Introduction}

For any $r>0$,
let $S^3(r) \subset \mathbb{R}^4$ denote the hypersphere
$$
x_1^2+x_2^2+x_3^2+x_4^2=r^2,
$$
of radius $r.$ (We set $S^3=S^3(1)$ for the unit hypersphere.)
In his letter \cite{sarnak} about the efficiency of a universal set of quantum gates, Sarnak has raised the 
question of how well one can approximate points on $S^3$ by rational and $S$-integral points of small height.

Consider the ball $B_\ve(\bxi)= \{\x\in \RR^4: 
\|\x-\boldsymbol{\xi}\|<\varepsilon\}$, for any $\ve>0$ and any  $\boldsymbol{\xi} \in S^3$, where 
$\|\cdot\|$ denotes  the Euclidean norm on $\RR^4$.
The spherical cap $S^3\cap B_\ve(\bxi)$ has volume
$\tfrac{4\pi}{3} \ve^3+O(\ve^5)$.
Given $r>0$ such that $r^2\in \ZZ$, 
we let 
$\lambda(r)$ denote the
maximal  volume 
of any  cap 
$S^3\cap B_\ve(\bxi)$, for
 $\bxi\in S^3$, 
which 
contains no points of the form $\x/r$, for $\x\in \ZZ^4$.
Sarnak then defines the {\em covering exponent}
to be 
\begin{equation}\label{eq:K}
K(S^3)=
\limsup_{r\to \infty} \frac{\log(\#S^3(r)\cap\ZZ^4)}{\log\left(
(\vol S^3)/\lambda(r)\right)
}.
\end{equation}
As is well-known, we have 
 $\vol S^3=2 \pi^2$ and 
$\#S^3(r)\cap\ZZ^4=c_r r^2(1+o(1))$, as $r\to \infty$,  
for an appropriate (slowly growing) function $c_r$ of $r$. 
According to \cite[Thm.~20.9]{IK}, we have 
$\log r \gg c_r\gg_\epsilon r^{-\epsilon}$, for any $\epsilon>0$,  as long as the largest power of 2 
dividing $r^2$ is bounded absolutely. 
In particular, the limit  in \eqref{eq:K} should  be 
understood as running over such $r$'s.

The ``big holes'' phenomenon, which is described in  \cite[Appendix 2]{sarnak}, shows that $K(S^3)\geq \frac{4}{3}$. 
Sarnak conjectures that this lower bound should be the truth, before using automorphic forms for $\mathrm{PGL}_2$  to show that $K(S^3)\leq 2$ in \cite[Appendix 1]{sarnak}.
This  upper bound was recovered  by Sardari \cite{S} by incorporating  Kloosterman's method into a smooth $\delta$-function  variant of the Hardy--Littlewood  method  due to  Duke, Friedlander and Iwaniec \cite{DFI}, and later 
developed 
extensively  by Heath-Brown \cite{HB}.
(Sardari's work is  actually much more general and, in fact, 
he obtains the optimal covering exponent $K(S^{n-1})=2-\frac{2}{n-1}$ for any $n>4$.)

Our main result establishes 
Sarnak's conjecture for $S^3$, under the assumption of 
a natural variant of  the Linnik conjecture about sums of Kloosterman sums. For any $m,n\in \ZZ$ and any $c\in \NN$, recall the definition
\begin{equation}\label{eq:kloosterman}
S(m,n;c)=\sum_{\substack{x\bmod{c}\\ (x,c)=1}} e_c(mx+n\bar{x}),
\end{equation}
of the Kloosterman sum, where $\bar{x}$ denotes the multiplicative inverse of $x$ modulo $c$.
We propose the following conjecture.

\begin{conjecture}[Twisted Linnik]\label{c:linnik} 
Let $B\geq 1$
and let  $m,n\in \ZZ$ be non-zero. 
Let $k\in \NN$ and let $a\in \ZZ/k\ZZ$. 
Then for any $\alpha\in [-B,B]$ we have
$$
\sum_{\substack{c \equiv a \mod{k}\\  c \leq X}}  \frac{S(m ,n;c)}{c}e\left(\frac{2\sqrt{mn}}{c} \alpha\right) 
 \ll_{\epsilon,k,B} (|mn| X)^{\epsilon},
 $$
for any $\epsilon>0$.
\end{conjecture}

For comparison, on invoking the triangle inequality,  it follows from  
 Weil's  bound for the Kloosterman sum (see \eqref{eq:weil}) that the left hand side has size $O_\epsilon(|mn|^\epsilon X^{\frac{1}{2}+\epsilon})$.
The usual {\em Linnik conjecture} corresponds to taking $\alpha=0$ in Conjecture \ref{c:linnik}.
The state of play concerning the case $\alpha=0$ is  discussed  in work of  Sarnak and Tsimerman \cite{ST}.  As evidence for Conjecture \ref{c:linnik}, Steiner \cite{steiner} has shown 
that the unconditional estimates achieved in \cite{ST} for $\alpha=0$ continue to hold for any $\alpha\in \RR$ such that $|\alpha|\leq 1-\delta$, for a fixed $\delta>0$. The case $|\alpha|>1-\delta$ is also discussed in \cite{steiner}, 
where the introduced twist cancels out the oscillatory behaviour of the Bessel functions, ultimately leading to slightly weaker estimates.
Unfortunately, the unconditional estimates obtained in \cite{steiner} are not sharp enough
to prove that that $K(S^3)<2$ unconditionally. 

Using  Sardari's work \cite{S} as a base, we shall establish the following result. 

\begin{theorem}\label{t:main}
Assume the twisted Linnik conjecture. Then $K(S^3)=\frac{4}{3}$.
\end{theorem}

The proof of this theorem is founded on exploiting extra cancellation in sums of the form
$$ 
\sum_{\substack{q\equiv 1\bmod{2}\\ q\leq Q}} 
q^{-2}S(r^2,c_1^2+c_2^2+c_3^2+c_4^2;q)
e_{q}(-2r \c.\boldsymbol{\xi})K_q(\c),
$$
for non-zero vectors $\c\in \ZZ^4$, where $K_q(\c)$ is a certain $4$-dimensional oscillatory integral that is 
revealed through an examination of  \eqref{eq:goaty} and \eqref{eq:4}. 
(There are similar expressions for $q\equiv \{0,2\}\bmod 4$.) 
Whereas Sardari brings the modulus sign inside, before invoking  Weil's bound to estimate the Kloosterman sum, our goal is take advantage of sign changes in it.
There are three key problems  in carrying out this plan.  

The first two problems arise when 
using partial summation to remove the factor $q^{-1}e_{q}(-2r\c.\boldsymbol{\xi})K_q(\c)$ .  For typical vectors $\c$, the 
derivative of 
$e_{q}(-2r\c.\boldsymbol{\xi})$ with respect to $q$ is very large.
This deficiency is what lies behind our  need to study sums of Kloosterman sums twisted by an exponential factor, as in 
Conjecture \ref{c:linnik}. 
Similarly, the  derivative   $\frac{\partial}{\partial q}K_q(\c)$ 
is also too large,  unless $q$ has exact order of magnitude  $Q$.  This presents our second problem.  
To circumvent this difficulty we shall use stationary phase to get an asymptotic expansion of $K_q(\c)$, to arbitrary precision, before using partial summation 
to rid ourselves of each term in the asymptotic expansion separately. 

Finally, consider the expression in the left hand side of Conjecture \ref{c:linnik}.
The third problem comes from a need for complete uniformity in $m$ and $n$ in any unconditional treatment of this sum. In fact, in the present situation, we are faced with the harder {\em Selberg range}, where $\sqrt{|mn|}>X$.  
Although Steiner  \cite{steiner} has achieved unconditional bounds
that go beyond the Weil bound in certain ranges, these fall  short of yielding an unconditional proof that $K(S^3)<2$.  Thus, in our work, we shall be content  with showing that 
the optimal covering exponent is a consequence of our twisted version of Linnik's conjecture.

\begin{remark}
As outlined by Sarnak  \cite{sarnak}, the study of $K(S^3)$ has its roots in the Solovay--Kitaev theorem in theoretical quantum computing. 
Consider the   single qubit gate set $S=\{s_1^{\pm },s_2^{\pm },s_3^{\pm }\}\subset \mathrm{SU}(2)$, 
where
$$
s_1=\frac{1}{\sqrt{5}}
\left( \begin{matrix}1+2i&0\\ 0&1-2i \end{matrix} \right),\quad
s_2=\frac{1}{\sqrt{5}}
\left( \begin{matrix}1&2i\\ 2i&1 \end{matrix} \right),\quad
s_3=\frac{1}{\sqrt{5}}
\left( \begin{matrix}1&2\\ -2&1 \end{matrix} \right).
$$
This  set is symmetric and  topologically dense in $\mathrm{SU}(2)$.
Sarnak defines a covering exponent $K(S)$,  which 
measures how efficiently the free group $\langle S\rangle $ generated by $S$  covers $\mathrm{SU}(2)$.  It follows from Theorem \ref{t:main} 
that $K(S)=\frac{4}{3}$  under the assumption of the  twisted Linnik conjecture.
\end{remark}

 \begin{ack}
 The authors are grateful to Peter Sarnak for his encouragement and to the anonymous referee for helpful comments.
While working on this paper the first author was
supported by ERC grant \texttt{306457}.  
\end{ack}

 \section{Preliminaries}

\subsection{Overview}

Let $r\in \NN$ such that the power of $2$ dividing $r$ is bounded absolutely. 
Let  $N=4r^2$.  
Fix a choice of $\boldsymbol{\xi} \in \RR^4$ such that $F(\boldsymbol{\xi})=1$,
where $F$ henceforth denotes   the non-singular quadratic form
$$
F(\x)=x_1^2+x_2^2+x_3^2+x_4^2.
$$
For any $\ve>0$, we let 
$$
S_\ve(N)=\left\{\x\in \ZZ^4: F(\x)=N, ~
\|\x/\sqrt{N}-\boldsymbol{\xi}\|<\varepsilon
\right\}.
$$
Our primary objective is to produce a lower bound on $\ve$, in terms of $N$, which is sufficient
to ensure that $S_\ve(N)$ is non-empty.
Sardari's work shows that $S_{\ve}(N)\neq \emptyset$ if $\ve\gg_\delta N^{-\frac{1}{6}+\delta}$, for any $\delta>0$.
This implies that $$
\lambda(r)\ll_\delta N^{-\frac{1}{2}+\delta}=(2r)^{-1+2\delta},
$$ 
for any $\delta>0$, whence  $K(S^3)\leq 2$ in \eqref{eq:K}.
Assuming Conjecture \ref{c:linnik}, 
we shall show that $S_{\ve}(N)\neq \emptyset$ if 
$\ve\gg_\delta N^{-\frac{1}{4}+\delta}$, for any $\delta>0$. This
implies that $\lambda(r)\ll_\delta r^{-\frac{3}{2}+2\delta}$,
whence $K(S^3)\leq \frac{4}{3}$, as required to complete the proof of Theorem  \ref{t:main}.

\subsection{Notation}\label{s:notation}

We denote by $\|\cdot\|$ the usual Euclidean norm, so 
that  $\|\x\|=\sqrt{F(\x)}$ 
 on $\RR^4$. 
 Throughout our work we  reserve  $\delta>0$ for  a small positive parameter.

One of the key innovations in Sardari's work \cite{S}
concerns the introduction of a new basis given by the tangent space of $F$ at $\bxi$ and we proceed to recall the construction here.
Let $\e_4=\vect{\xi}$. (This  is the unit vector in the direction of $\nabla F(\boldsymbol{\xi})=2\bxi$.)
Choose an orthonormal basis $\e_1,\e_2,\e_3$ for the tangent space 
$T_{\boldsymbol{\xi}}(F) = \e_4^{\perp}$.
Recalling that $F(\bxi)=1$, it therefore follows that 
$$
F(u_1\e_1+\dots+u_4\e_4)=F(\u),
$$
for any  $\u\in \RR^4$. Finally, any vector $\b\in \RR^4$ can be written $\b=\sum_{i=1}^4 \hat b_i \e_i$, with $\hat b_i=\b.\e_i$, for $1\leq i\leq 4$.

\subsection{Activation of the circle method}

We begin by choosing  a smooth   function $w_0:\RR\to \RR_{\geq 0}$ with unit mass, such that 
$
\supp(w_0)=[-1,1].$ 
  We will  work with the weight function
$w:\RR^4\to \RR_{\geq 0}$, given by
\begin{equation}\label{eq:ww}
w(\x) =
w_0\left(\frac{\|\x-\bxi\|}{\varepsilon}\right)
w_0\left(\frac{2\bxi . (\boldsymbol{x-\xi})}{\varepsilon^2}\right).
\end{equation}
Let
$$
\Sigma(w) = \sum_{\substack{\x\in \ZZ^4\\ F(\x)=N}} w\left(\frac{\x}{\sqrt{N}}\right),
$$
for any $N\in 4\NN$. 
We want conditions on $\ve$, in terms of $N$, under which  $\Sigma(w)>0$. Indeed, 
if $\Sigma(w)>0$, then 
there exists a vector $\x\in \ZZ^4$
such that $F(\x)=N$ and 
$$
\|\x/\sqrt{N}-\bxi\|<\ve, \quad
|2\bxi.(\x/\sqrt{N}-\bxi)|<\ve^2.
$$
It follows from Sardari's argument that $\Sigma(w)>0$ if $\ve\gg_\delta N^{-\frac{1}{6}+\delta}$,  for  any $\delta>0$. Our goal is to draw the same conclusion provided that 
$\ve\gg_\delta N^{-\frac{1}{4}+\delta}$.

A few words are in order regarding the  inequality 
$|2\bxi.(\x/\sqrt{N}-\bxi)|<\ve^2$ 
that is enshrined in our counting function $\Sigma(w)$.
Suppose that $\|\x/\sqrt{N}-\bxi\|<\ve$. Then we may write 
 $\x/\sqrt{N} =\bxi + \varepsilon \z$, with $\|\z\| < 1.$ 
Under this change of variables, the inequality $|2\bxi.(\x/\sqrt{N}-\bxi)|<\ve^2$  is  equivalent to 
$|2\bxi.\z|<\ve$, and 
\begin{equation*}
F(\x) - N = N\left(2\varepsilon \bxi.\z + \varepsilon^2 F(\z)\right).
\end{equation*}
Thus, we must have 
$|2\bxi.\z|<\ve$
when the left hand side vanishes. Moreover it is clear that 
$F(\x) - N \ll \varepsilon^2 N$
for any $\x$ such that $w(\x/\sqrt{N})\neq 0$.

One ``level lowering'' effect of this is that we are allowed to take 
$$
Q=\ve\sqrt{N}
$$ 
in the version of the circle method 
recorded by Heath-Brown
 \cite[Thm.~2]{HB}, rather than $Q=\sqrt{\ve N}$, as might at first appear.
We conclude that there exists a constant  
$c_Q=1+O_A(Q^{-A})$, for any $A>0$, such that
\begin{equation}\label{eq:2}
\Sigma(w) = \frac{c_Q}{Q^2}\sum_{q=1}^{\infty}\sum_{\c \in \mathbb{Z}^4}q^{-4}S_q(\c)I_q(\c),
\end{equation} 
where
\begin{equation}\label{eq:3}
\begin{split}
S_q(\c) &= \sumstar_{\substack{a \mod{q}}} \sum_{\b \mod{q}} e_q\left(a\left\{F(\b) -N\right\}+ \b.\c\right), \\
I_q(\c) &= \int_{\RR^4} w\left(\frac{\x}{\sqrt{N}}\right) h\left(\frac{q}{Q},\frac{F(\x)-N}{Q^2}\right)e_q(-\c.\x) \, \d\x.
\end{split}
\end{equation}
Here $h:(0,\infty)\times\RR\to \RR$ is a certain function such that $h(x,y)\ll x^{-1}$ for all $y$ and 
$h(x,y)=0$ unless $x\leq \max\{1,2|y|\}$.
In particular, only values of  $q\ll Q$ contribute to  $\Sigma(w)$ in \eqref{eq:2}.
Thus, in all that  follows, we may henceforth assume that $Q\geq 1$; viz. $\ve^{-1}\leq \sqrt{N}$.

We shall prove that 
Conjecture \ref{c:linnik} implies 
$\Sigma(w)>0$ if $\ve\gg_\delta N^{-\frac{1}{4}+\delta}$, for any $\delta>0$. In fact we shall establish an asymptotic formula for $\Sigma(w)$, in which the main term involves a pair of constants $\sigma_\infty$ and $\mathfrak{S}$.
The constant  $\sigma_\infty$ is equal to the weighted real density of points on $S^3$ and is 
given explicitly  in \eqref{eq:sigma}.  The constant $\mathfrak{S}$ is the usual product of non-archimedean local densities, with value 
\begin{equation}\label{eq:density}
\mathfrak{S}=\prod_p \sigma_p,  \quad \sigma_p=\lim_{k\to \infty} p^{-3k}\#\{\x\in (\ZZ/p^k\ZZ)^4: F(\x)\equiv N \bmod{p^k}\}.
\end{equation}
We may now record our main result.

\begin{theorem}\label{thm1}
Assume Conjecture \ref{c:linnik}. Then, for any $\delta>0$, we have
$$
\Sigma(w)=\frac{\ve^3N \sigma_\infty \mathfrak{S}}{2}
+O_\delta\left(
\ve^4N^{1+\delta}+
\ve^{\frac{5}{2}}N^{\frac{3}{4}+\delta}+\varepsilon N^{\frac{1}{2}+\delta}\right).
$$
\end{theorem}

We shall see that  $\sigma_\infty\gg 1$ in \eqref{eq:sigma}. Likewise, 
as remarked upon by Sardari \cite[Remark 1.4]{S}, we have $\mathfrak{S}\gg_\delta N^{-\delta}$ for any $\delta>0$,
if the power of $2$ dividing $N$ is bounded. 
 Thus Theorem \ref{thm1} implies Theorem \ref{t:main}.

The remainder of the paper is as follows. 
In \S \ref{s:sum} we shall explicitly evaluate the sum $S_q(\c)$ using
Gauss sums. Next, in \S \ref{s:integral}, we shall study the oscillatory integrals $I_q(\c)$ using stationary phase. Finally, in \S \ref{s:final}, we shall combine the various estimates and complete the proof of Theorem \ref{thm1}.

\section{Gauss sums and Kloosterman sums}\label{s:sum}

In this section we 
explicitly evaluate the exponential sum 
$S_q(\c)$ in \eqref{eq:3}, for $\c\in \ZZ^4$ and relate it to the Kloosterman sum
$S(m,n;c)$ in \eqref{eq:kloosterman}. The latter sum satisfies the well-known Weil bound 
\begin{equation}\label{eq:weil}
|S(m,n;c)|\leq \tau(c)\sqrt{(m,n,c)}\sqrt{c},
\end{equation}
where $\tau$ is the divisor function. 

Recalling that $N\in 4\NN$, it will be convenient to write $N=4N'$ for $N'\in \NN$.
We have
\begin{equation}\label{eq:Sq}
S_q(\c) = \sumstar_{a\bmod{q}} e_q(-4aN')\prod_{i=1}^4 \cG(a,c_i;q),
\end{equation}
where
$$
\cG(s,t;q)=\sum_{b \bmod{q}}e_q\left(sb^2+tb\right),
$$
for given non-zero integers $s,t,q$ such that $q\geq 1$.
The latter sum is classical and may be evaluated. 
Let 
$$
\delta_n=\begin{cases} 0, & \mbox{ if }n\equiv 0\bmod{2}, \\ 1 & \mbox{ if } n\equiv 1\bmod{2}, \end{cases}
\quad
\epsilon_n=\begin{cases} 1, & \mbox{ if }n\equiv 1\bmod{4}, \\ i, & \mbox{ if } n\equiv 3\bmod{4}.
\end{cases}
$$
The following result is recorded in \cite[Lemma 3]{BB}, but goes back to Gauss.

\begin{lemma} \label{gaussevlemma} 
Suppose that $(s,q)=1$. Then 
$$ 
\mathcal{G}(s,t;q)=
\begin{cases}
\epsilon_q \sqrt{q}  \left(\frac{s}{q}\right) e\left(-\frac{\overline{4s}t^2}{q}\right) & \text{ 
if $q$ is odd,}\\
2  \delta_t \epsilon_v\sqrt{v} \left(\frac{2s}{v}\right) e\left(-\frac{\overline{8s}t^{2}}{v}\right) &
\text{ if $q=2v$, with $v$ odd,}\\
(1+i) \epsilon_s^{-1} (1-\delta_t) \sqrt{q}\left(\frac{q}{s}\right) e\left(-\frac{\overline{s}t^{2}}{4q}\right) 
& \text{ 
if $4\mid q$.}
\end{cases}
$$
\end{lemma}

Our analysis of $S_q(\c)$ now differs according to the $2$-adic valuation of $q$.
In each case we shall be led to an appearance of the Kloosterman sum \eqref{eq:kloosterman}.

Suppose first that  $q\equiv 1\bmod{2}$.
Substituting Lemma \ref{gaussevlemma} into \eqref{eq:Sq} we directly obtain
$$
S_q(\c) = q^2\sumstar_{a\bmod{q}} e_q(-4aN'-\overline{4a}F(\c))
=q^2S(N',F(\c);q),
$$
since $S(A,tB;q)=S(tA,B;q)$ for any $t\in (\ZZ/q\ZZ)^*$.

If $q\equiv 2\bmod{4}$ then we write $q=2v$, for odd $v$. This time we obtain 
\begin{align*}
S_q(\c) 
&= 2^4\delta_{c_1c_2c_3c_4}v^2
\sumstar_{a\bmod{q}} e_q(-4aN')e_v(-\overline{8a}F(\c))\\
&= 4\delta_{c_1c_2c_3c_4}q ^2
S(N',F(\c)/4;v)\\
&= 4\delta_{c_1c_2c_3c_4}q ^2
S(2N',F(\c)/2;q), 
\end{align*}
since $4\mid F(\c)$, when all the $c_i$ are odd. 

If  $q\equiv 0\bmod{4}$, it  follows from Lemma \ref{gaussevlemma} that
\begin{align*}
S_q(\c) 
&= -4(1-\delta_{c_1})\dots (1-\delta_{c_4})q^2
~\sumstar_{a\bmod{q}} ~e_q(-4aN')e_{4q}(-\overline{a}F(\c)).
\end{align*}
Thus, in this case,  we find that 
$$
S_q(\c) =
\begin{cases}
0 & \text{ if $2\nmid \c$},\\
-4q^2 S(N,F(\c');q) &
\text{ if $\c=2\c'$ for $\c'\in \ZZ^4$.}
\end{cases}
$$

\section{Oscillatory  integrals}\label{s:integral}

Recall the definition 
\eqref{eq:3} of 
$I_q(\c)$, in which $w$ is given by \eqref{eq:ww}. 
We make the change of variables $\x=\sqrt{N}\x'$ and
$\x' = \boldsymbol{\xi} + \varepsilon\z$. This leads to the expression
\begin{align*}
I_q(\c) 
&= 
N^2 \int_{\RR^4} w\left(\x'\right) h\left(\frac{q}{Q},\frac{F(\x')-1}{\varepsilon^2}\right)e_{\frac{q}{\sqrt{N}}}(-\c.\x') \, \d\x'\\
&=
\ve^4N^2 e_{\frac{q}{\sqrt{N}}}(-\c.\boldsymbol{\xi}) \int_{\RR^4} w_0(\|\z\|)w_0\left(\frac{2\boldsymbol{\xi}.\z}{\varepsilon}\right)h\left(\frac{q}{Q},\frac{y(\z)}{\varepsilon}\right)e_{\frac{q}{\varepsilon\sqrt{N}}}(-\c.\z) \, \d\z,
\end{align*} 
where
$y(\z) = 2\bxi.\z+ \varepsilon F(\z)$.
Let $r=q/Q$ and $\v = r^{-1}\c$. Then we have 
\begin{equation}\label{eq:goaty}
I_q(\c) =\ve^4N^2 e_{r}(-\varepsilon^{-1}\c.\boldsymbol{\xi})I_r^{*}(\v),
\end{equation}
where
\begin{equation}\label{eq:4}
I_r^{*}(\v) = \int_{\RR^4} w_0(\|\x\|)w_0\left(\frac{2\bxi.\x}{\varepsilon}\right) h\left(r,\frac{y(\x)}{\varepsilon}\right)e(-\v.\x) \, \d\x.
\end{equation}
In particular, we have
$I_r^*(\v)=O(\ve/r)$, since  $h(r,y)\ll r^{-1}$ and the region of integration has measure $O(\ve)$.

\subsection{Easy estimates}
Our attention now shifts to analysing  $I_r^{*}(\v)$ for $r\ll 1$ and $\v\in \RR^4$. 
Let $\x\in \RR^4$ such that 
$w_0(\|\x\|)w_0(2\bxi.\x/\varepsilon)\neq 0$. Then 
$$
\frac{y(\x)}{\ve}=\frac{2\bxi.\x+\ve F(\x)}{\ve}<2.
$$
Put $v(t)=w_0(t/6)$. Then $v(y(\x)/\ve)\gg 1$ whenever
$w_0(\|\x\|)w_0(2\bxi.\x/\varepsilon)\neq 0$.  
We may now write
$$
I_r^{*}(\v) = \frac{1}{r}\int_{\RR^4} w_3(\x) f\left(\frac{y(\x)}{\varepsilon}\right) e(-\v.\x) \, \d\x,
$$
where $f(y)=v(y)rh(r,y)$ and 
\begin{equation}\label{eq:w3}
w_3(\x)=\frac{w_0(\|\x\|)w_0(2\bxi.\x/\varepsilon)}{v(y(\x)/\ve)}.
\end{equation}
Let $p(t)=\hat f(t)$ be the Fourier transform of $f$.  Then the proof of  \cite[Lemma~17]{HB} 
shows that  
\begin{equation}\label{eq:england_sucks}
p(t) \ll_j r(r|t|)^{-j},
\end{equation}
for any $j>0$. We may therefore write
\begin{equation}\label{eq:irv}
I_r^{*}(\v) = \frac{1}{r}\int_{\RR}p(t) \int_{\RR^4}w_3(\x)e\left(t\frac{y(\x)}{\varepsilon} - \v.\x \right)\, \d\x \, \d t.
\end{equation}
Building on this, we proceed by  establishing the following result. 

\begin{lemma}\label{largec}
Let $\c \in \ZZ^4$, with $\c\neq \0$. 
Then
\begin{equation*}
I_q(\c) \ll_{j} \frac{\ve^5N^2 Q}{q}  \min_{i=1,2,3}\left\{|\hat c_i|^{-j},(\varepsilon |\hat c_4|)^{-j} \right\},
\end{equation*}
for any $j>0$.
\end{lemma}

This result corresponds to  \cite[Lemma 6.1]{S}. 
Since $\max_i |\hat c_i|\gg \|\c\|$, it follows that 
$$I_q(\c) \ll_{j} \frac{\ve^5N^2 Q}{q} (\varepsilon \|\c\|)^{-j},
$$
for any $j>0$. In this way, for any $\delta>0$,  Lemma \ref{largec} implies that 
there is a negligible contribution to ~\eqref{eq:2}  from $\c$ such that 
either  of the inequalities 
$\|\c\| > N^{\delta}/\varepsilon$ or 
$\max_{i=1,2,3}\left\{|\hat c_i|,\varepsilon |\hat c_4| \right\} >N^{\delta}$ hold. Thus, in \eqref{eq:2},  the summation over $\c $ can henceforth be restricted to  the set $\mathcal{C}$, which is defined to be the set of $\c\in \ZZ^4$ for which 
$\|\c\| \leq  N^{\delta}/\varepsilon$  and 
$\max_{i=1,2,3}\left\{|\hat c_i|,\varepsilon |\hat c_4| \right\} \leq N^{\delta}.$ 
It follows from  \cite[Lemma 6.3]{S} that $\#\mathcal{C}=O( \ve^{-1} N^{4\delta})$.

\begin{proof}[Proof of Lemma \ref{largec}]
We make the change of variables $\x = \sum_{i=1}^4 u_i\e_i$ in \eqref{eq:irv}.  In the notation of \S \ref{s:notation},  
let  $\v=\sum_{i=1}^4\hat v_i\e_i$, where $\hat v_i=\v.\e_i$.
Then, on recalling \eqref{eq:w3}, we find that  
\begin{equation*}
\begin{split}
I_r^{*}(\v) &= \frac{1}{r} \int_{\RR}p(t) \int_{\RR^4}w_3\left(\sum_{i=1}^4u_i\e_i\right)
e\left(\frac{ty(\sum_{i=1}^4u_i\e_i)}{\varepsilon} - 
\u.\hat\v\right)\, \d\u \, \d t \\
&= \frac{1}{r} \int_{\RR}p(t) \int_{\RR^4}
\frac{w_0(\|\u\|)w_0( 2u_4/\ve)}{v ((2u_4+\ve F(\u))/\ve)}
e\left(H(\u)\right)\, \d\u \, \d t,
\end{split}
\end{equation*}
where
$
H(\u)=\frac{t}{\ve}\left\{ 2u_4+\ve F(\u)\right\} -   \u.\hat\v.
$
We have
$$
\frac{\partial H(\u)}{\partial u_i} = 
 \begin{cases} 
2tu_i - \hat v_i &\text{ if $1\leq i\leq 3$,}\\
2tu_4 - \hat v_4+\frac{2t}{\ve} &\text{ if $i=4$.}
\end{cases}
$$
The proof of the lemma now follows from repeated integration by parts in conjunction with \eqref{eq:england_sucks}, much as in the proof of \cite[Lemma 19]{HB}. Thus, when $i\in \{1,2,3\}$, integration by parts with respect to $u_i$ readily yields
\begin{equation*}
I_r^{*}(\v) \ll_j \frac{\ve}{r} \left\{r  |\hat v_i|^{1-j} +  r^{1-j}|\hat v_i|^{1-j}\right\} \ll_j \ve r^{-j} |\hat v_i|^{1-j},
\end{equation*}
for any $j>0$, since $r\ll 1$.  Likewise, integrating by parts with respect to $u_4$, we get
\begin{equation*}
I_r^{*}(\v) \ll_j \frac{\ve}{r} \left\{r (\varepsilon |\hat v_4|)^{1-j} +  r^{1-j}(\varepsilon |\hat v_4|)^{1-j}\right\}
\ll_j \ve r^{-j} (\ve |\hat v_4|)^{1-j}.
\end{equation*} 
The statement of the lemma follows on recalling \eqref{eq:goaty} and the fact  that  $\c = r\v$, with $r=q/Q$.
\end{proof}

\subsection{Stationary phase}

The following stationary phase result will prove vital  in our more demanding analysis of $I_q(\c)$ in the next section. 

\begin{lemma}\label{lem:phase} 
 Let $\phi$ be a Schwartz function on $\BR^n$ and let $N\geq 0$.  Then
\begin{align*}
\int_{\BR^n} e^{i \lambda \|\vect{x}\|^2}\phi(\x) \d\x =& \lambda^{-\frac{n}{2}} \sum_{j=0}^N a_j \lambda^{-j}+O_{n,N}\left(|\lambda|^{-\frac{n}{2}-N-1}
\|\phi\|_{2N+3+n,1}
\right),
\end{align*}
where $\|\cdot\|_{k,1}$ denotes the Sobolev norm on $L^1(\RR^n )$ 
of order $k$ and 
$$
a_j=(i\pi)^{\frac{n}{2}} \frac{i^j}{j!} \left( \Delta^j \phi \right)(\vect{0}).
$$

\end{lemma}

\begin{proof} 
We  follow the argument in Stein  \cite[\S VIII.5.1]{stein}.
Using the Fourier transform, we can write the  integral as
\begin{equation}\label{eq:bluemat}
\left( \frac{i\pi}{\lambda}\right)^{\frac{n}{2}} \int_{\BR^n} e^{-i \pi^2 \|\bxi\|^2/\lambda} \widehat{\phi}(\bxi) \d\bxi.
\end{equation}
Next, we split off the first $N$ terms in a Taylor expansion around $\vect{0}$, finding that 
$$
e^{-i \pi^2 \|\bxi\|^2/\lambda} = \sum_{j=0}^N \frac{(-i\pi^2 \|\bxi\|^2/\lambda)^j}{j!}+R_N(\bxi).
$$
The main term now comes from integration by parts and Fourier inversion. We are left to deal with the integral involving $R_N(\vect{\xi})$. We have
\begin{equation}
R_N(\vect{\xi})\ll_N\left(\frac{\|\bxi\|^2}{|\lambda|}\right)^{N+1},
\label{eq:taylorest}
\end{equation}
which follows from Taylor expansion when $\|\vect{\xi}\|^2\le |\lambda|$ 
and trivially otherwise. Moreover,   
\begin{equation}
\widehat{\phi}(\bxi) = O_{A}\left( \|\vect{\xi}\|^{-A} \|\phi\|_{A,1} \right), 
\label{eq:Fourierest}
\end{equation}
for any $A\geq 0$.
We split up the remaining integral into two parts: $\|\vect{\xi}\| \le 1$ and $\|\vect{\xi}\| >1$. For the first part we use \eqref{eq:taylorest} and \eqref{eq:Fourierest} with $A=2N+1+n$. Recalling the additional factor $\lambda^{-\frac{n}{2}}$ from \eqref{eq:bluemat}, we get an error term of size
$$
O_{n,N} \left( |\lambda|^{-\frac{n}{2}-N-1} \|\phi\|_{2N+1+n,1} \right). 
$$
For the second part we use \eqref{eq:taylorest} and \eqref{eq:Fourierest}, but this time with  $A=2N+3+n$. This leads to the same overall error term, but with 
the factor $\|\phi\|_{2N+1+n,1}$ replaced by  
$\|\phi\|_{2N+3+n,1}$. 
\end{proof}

\subsection{Hard estimates}

Having shown how to truncate the sum over $\c$ in \eqref{eq:2}, we now return to 
 \eqref{eq:goaty} for $\c\in \mathcal{C}$ 
 and see what more can be said about the integral $I_r^*(\v)$
 in \eqref{eq:4},   with $r=q/Q$ and $\v=r^{-1}\c$.
Our result relies on an asymptotic expansion of $I_r^*(\v)$, but the form it takes  depends on the size of $\ve|\hat v_4|$. 

It will be convenient to set 
$
\a=(\hat v_1,\hat v_2, \hat v_3),
$
in what follows.  To begin with,
we make the change of variables $\x = \sum_{i=1}^4 u_i\e_i$ in \eqref{eq:4}. This leads to the expression
\begin{equation*}
\begin{split}
I_r^*(\v) &=  \int_{\RR^4}
w_0(\|\u\|)
w_0( 2u_4/\ve)
h\left(r, 
\frac{2u_4}{\ve}+F(\u)\right)
e(-\u.\hat\v) \, \d\u,
\end{split}
\end{equation*}
where $\hat v_i=\v.\e_i$ for $1\leq i\leq 4$.
We now write  $y=2u_4/\ve+F(\u)$, under which we have 
\begin{equation}\label{eq:u4}\begin{split}
u_4&=\frac{1}{\ve}\left(-1+\sqrt{1+\ve^2\{y-u_1^2-u_2^2-u_3^2\}}\right).
\end{split}\end{equation}
Thus 
\begin{equation}\label{eq:m4}
I_r^*(\v) = \int_\RR h(r,y) 
e\left(-\frac{\ve \hat v_4 y}{2}\right)  
T(y) \d y,
\end{equation}
where 
\begin{equation}\label{eq:m4'}
T(y)= 
e\left(\frac{\ve \hat v_4 y}{2}\right)  
\int_{\RR^3}
w_0(\|\u\|)
w_0( 2u_4/\ve)
e(-\u.\hat\v) \, \frac{\d u_1\d u_2 \d u_3}{2/\ve +2u_4},
\end{equation}
and $u_4$ is given in terms of $y,u_1,u_2,u_3$ by \eqref{eq:u4}.
In particular, on writing $\x=(u_1,u_2,u_3)$, we have 
$w_0(\|\u\|)
w_0( 2u_4/\ve)=\psi_y(\x)$, where
$\psi_y:\RR^3\to \RR_{\geq 0}$ is the weight function
\begin{equation}\label{eq:psi}
\begin{split}
\psi_y(\x)=~&
w_0\left(2\ve^{-2}(
-1+\sqrt{1+\ve^2\{y-\|\x\|^2\}})\right)\\
&\times 
w_0\left(\sqrt{\|\x\|^2+
\ve^{-2}(1-\sqrt{1+\ve^2\{y-\|\x\|^2\}})^2}\right).
\end{split}
\end{equation}
We note, furthermore,  that the integral in $T(y)$
 is supported on $[-1,1]^3$.
Moreover, we have
\begin{equation}\label{eq:approximate}
\frac{2u_4}{\ve}=
\frac{2}{\ve^2}\left(-1+\sqrt{1+\ve^2\{y-\|\x\|^2\}}\right) =y-\|\x\|^2+O(\ve^2),
\end{equation}
for any $\x$ such that $\psi_y(\x)\neq 0$.
In particular, it follows that 
\begin{equation}\label{eq:m4''}
\frac{1}{2/\ve +2u_4}=\frac{\ve}{2}\left(1+O(\ve^2)\right)
\end{equation}
in \eqref{eq:m4'}.

Since $e(z)=1+O(z)$, we invoke  \eqref{eq:u4} and \eqref{eq:approximate}
to deduce that 
\begin{equation}\label{eq:egg}
e(-\u.\hat\v)=
e\left(-\frac{\ve\hat v_4 y}{2}\right)  
e\left(\frac{\ve\hat v_4}{2}\|\x\|^2-\a.\x\right)   \left(1+O(|\ve\hat v_4|\ve^2)\right),
\end{equation}
where we recall that  $\a=(\hat v_1,
\hat v_2, \hat v_3)$.
Thus, it follows from  \eqref{eq:m4''} that
\begin{equation}\label{eq:neddy}
\begin{split}
T(y)&=
\frac{\ve}{2}\left(1+O(\ve^2+|\ve\hat v_4|\ve^2)\right) 
I(y),
\end{split}
\end{equation}
where
\begin{equation}\label{eq:Iy}
I(y)=
\int_{\RR^3}
\psi_y(\x)
e\left(\frac{\ve\hat v_4}{2} \|\x\|^2 -\a.\x\right)
\d \x.
\end{equation}

In what follows it will be useful to record the estimate
\begin{equation}\label{eq:seren}
\int_\RR \left|r^k y^{\ell} \frac{\partial^k h(r,y)}{\partial r^k}\right|\d y\ll_{\ell} r^{\ell}, 
\end{equation}
for any $\ell\geq 0$
and $k\in \{0,1\}$.
This is a straightforward consequence of 
\cite[Lemma 5]{HB}.
The stage is now set to prove the following preliminary estimate for $I_r^*(\v)$ and its partial derivative with respect to $r$.

\begin{lemma}\label{lem:jumper}
Let   $k\in \{0,1\}$.
 Then
$$
r^{2k} \frac{\partial^k I_r^*(\v)}{\partial r^k} \ll  \frac{\ve(1+\ve^3|\hat v_4|)}{\max\{1,(\ve|\hat v_4|)\}^{\frac{3}{2}}} N^{\delta}.
$$
\end{lemma}
\begin{proof}
Suppose first that $k=0$. An application of    \cite[Lemmas 3.1 and 3.2]{HBP} shows that  
$$
I(y)\ll 
 \frac{1}{\max\{1,(\ve|\hat v_4|)\}^{\frac{3}{2}}},
$$
since $\|\hat \psi_y\|_1\ll 1$. 
The desired bound now follows on substituting this into 
\eqref{eq:m4} and 
\eqref{eq:neddy}, before using \eqref{eq:seren} with $k=\ell=0$ to carry out the integration over $y$.

Suppose next that $k=1$. 
Then,  in view of \eqref{eq:m4}, 
 we have
\begin{equation}\label{eq:runrabbit}
\begin{split}
r^{2} \frac{\partial I_r^*(\v)}{\partial r} =~& 
\int_{\mathbb{R}} r^{2} \frac{\partial h(r,y)}{\partial r} e\left(-\frac{\ve \hat v_4 y}{2}\right)  
T(y) \d y \\ 
&+\int_{\mathbb{R}} h(r,y) e\left(-\frac{\ve \hat v_4 y}{2}\right)  \widetilde{T}(y) \d y, \\   
\end{split}
\end{equation}
where
\begin{equation*}
\begin{split} 
\widetilde{T}(y) &= e\left(\frac{\ve \hat v_4 y}{2}\right) \int_{\RR^3}
w_0(\|\u\|)
w_0( 2u_4/\ve)
 r^{2}\frac{\partial}{\partial r}{e(-\u.\hat\v)} \, \frac{\d u_1\d u_2 \d u_3}{2/\ve +2u_4} \\
 &=  
 e\left(\frac{\ve \hat v_4 y}{2}\right) \int_{\RR^3}
(2\pi i \u.\hat\c) 
w_0(\|\u\|)
w_0( 2u_4/\ve)
e(-\u.\hat\v)  \frac{\d u_1\d u_2 \d u_3}{2/\ve +2u_4}.
\end{split}
\end{equation*}
The contribution from the first integral in \eqref{eq:runrabbit}
is satisfactory, since $r\ll 1$, 
on reapplying our argument for $k=0$ and using \eqref{eq:seren} with $k=1$ and $\ell=0$.
Turning to the second integral in \eqref{eq:runrabbit}, we 
recall 
\eqref{eq:m4''} and \eqref{eq:egg}. These allow us to  write 
$$
\widetilde{T}(y) = \ve\pi i \left(1+O(\ve^2+|\ve\hat v_4|\ve^2)\right) 
\widetilde{I}(y),
$$
where
$$
\widetilde{I}(y) = \int_{\mathbb{R}^3}\widetilde{\psi_y}(\x)e\left(\frac{\ve\hat v_4}{2} \|\x\|^2 -\a.\x\right)
\d \x
$$
and 
$$
\widetilde{\psi_y}(\x) = \left(r \a.\x + \frac{\hat c_4}{\ve}\left(-1+\sqrt{1+\ve^2\{y-\|\x\|^2}\}\right)\right)\psi_y(\x).
$$
Here, the  definition  of $\mathcal{C}$ implies that $r|\a|=\max\{|\hat c_1|, |\hat c_2|,|\hat c_3|\}\leq N^\delta$
and $\ve|\hat c_4|\leq N^\delta$. 
Thus the $L^1$-norm of the Fourier transform of 
$\widetilde{\psi_y}$ is $O( N^{\delta}).$ 
Once combined with \eqref{eq:seren} with $k=\ell=0$, 
we apply 
   \cite[Lemmas~3.1 and 3.2]{HBP}  
to estimate $\widetilde I(y)$, which concludes our  treatment of the case $k=1$.
\end{proof}

The case $k=0$ of 
Lemma \ref{lem:jumper} is already implicit in Sardari's work (see \cite[Lemma 6.2]{S}). We shall also need the case $k=1$, but it turns out that it is only effective when $r$ is essentially of size $1$. For general $r$, we require 
a pair of  asymptotic expansions  for
$I_r^*(\v)$,    that are relevant for small and large values of  $\ve|\hat v_4|$, respectively. This is the objective of the following pair of results.

\begin{lemma}\label{lem:I-small}
Let $A\geq 0$. Then 
\begin{align*}
I_r^*(\v)=~&
\frac{\ve I(0)}{2}
+O_A\left(
\ve^3(1+\ve|\hat v_4|) +\ve(1+\ve|\hat v_4|)^{A}r^{A}\right).
\end{align*}
\end{lemma}

\begin{proof}
Our first approach is founded on the Taylor expansion
$$
e\left(-\frac{\ve \hat v_4 y}{2}\right)  =\sum_{j=0}^{A-1} \frac{(
-\pi i \ve \hat v_4 y)^j}{j!} +R_A(y), $$
where $R_A(y)\ll_A (\ve |\hat v_4 y|)^{A}$. 
Since $I(y)\ll 1$,  we  conclude from \eqref{eq:m4}, \eqref{eq:neddy} 
and \eqref{eq:seren}
that 
\begin{align*}
I_r^*(\v)=~&
\frac{\ve}{2}
\sum_{j=0}^{A-1} \frac{(
-\pi i \ve \hat v_4)^j}{j!} 
\int_\RR  y^j h(r,y) I(y)\d y  \\
&+O_A\left(
\ve^3(1+\ve|\hat v_4|) +\ve(\ve|\hat v_4|)^{A}r^{A}\right).
\end{align*}
Next, we claim that 
\begin{equation}\label{eq:sheep}
\int_\RR  y^j h(r,y) I(y)\d y  =O_A(r^A)+
\begin{cases}
I(0) &\text{ if $j=0$,}\\
0 &\text{ if $j>0$.}
\end{cases}
\end{equation}
To see this, note that $I(y)$ belongs to the class of weight functions considered in \cite[Lemma 9]{HB}. This settles \eqref{eq:sheep} when $j=0$. When $j>0$ we 
truncate the integral to $|y|\leq \sqrt{r}$ and 
expand $I(y)$ as  a Taylor series, before invoking \cite[Lemma 8]{HB},
as in the proof of \cite[Lemma 9]{HB}.
This settles \eqref{eq:sheep} when $j>0$.
The statement of the lemma is now obvious.
\end{proof}

\begin{lemma}\label{lem:I-stationary}
Assume that $\ve|\hat v_4|>1$.
For each $j\geq 0$, we define
$$\phi_j(y)=\Delta^j
\psi_y\left((\ve \hat v_4)^{-1}\a\right)=\Delta^j\psi_y\left((\ve \hat c_4)^{-1}(\hat c_1,\hat c_2, \hat c_3)\right),
$$
where $\psi_y$ is given by \eqref{eq:psi}.
Let $A\geq 0$. 
Then
 there exist constants $k_j$ that depend only on $j$ such that 
\begin{align*}
I_r^*(\v)=~&
\frac{\ve \delta(\hat \c)}{
(\ve\hat v_4)^{\frac{3}{2}}}
e\left(-\frac{\|\a\|^2}{2\ve \hat v_4}\right) \sum_{j=0}^A  \frac{k_j}{(\ve\hat v_4)^j}
\int_\RR h(r,y) e\left(-\frac{\ve \hat v_4 y}{2}\right) 
\phi_j(y)  \d y\\
&+O_{A}\left(\frac{\ve^3}{|\ve\hat v_4|^{\frac{1}{2}}}+\frac{\ve}{ |\ve\hat v_4|^{\frac{5}{2}+A}}\right),
\end{align*}
where
$$
\delta(\hat \c)=\begin{cases}
1 &\text{ if $ \ve|\hat c_4|\gg |(\hat c_1,\hat c_2,\hat c_3)|$,}\\
0 &\text{ otherwise}.
\end{cases}
$$
\end{lemma}

\begin{proof}
It will be convenient to set 
 $\lambda= \ve\hat v_4$ in the proof of this result,
 recalling our hypothesis that $|\lambda|>1$. 
  Our starting point is the expression for $T(y)$ in \eqref{eq:neddy}, in which $I(y)$ is given by \eqref{eq:Iy}.
By completing the square, we may write
$$
T(y)
=\frac{\ve}{2}
\left(1+O(|\lambda|\ve^2)\right)  
e\left(-\frac{\|\a\|^2}{2\lambda}\right)
I^*(y),
$$
since $|\lambda|>1$, 
where
$$
I^*(y)
=\int_{\RR^3}
\psi_y\left(\x+\frac{\a}{\lambda}\right)
e\left(\frac{\lambda}{2} \|\x\|^2 \right)
\d \x.
$$

If  $|\a|\gg \ve|\hat v_4|$, then it follows from \cite[Lemma 10]{HB} that 
$
T(y)\ll_A \ve |\lambda|^{-A} ,
$
for any $A\geq 0$.
Alternatively, if $|\a|\ll  \ve|\hat v_4|$, which is equivalent to $\delta(\hat \c)=1$, then all the hypotheses of Lemma~\ref{lem:phase} are met. 
Thus,  for any $A\geq 0$, 
 there exist constants $k_j$ that depend only on $j$ such that 
$$
I^*(y)=
\frac{1}{
\lambda^{\frac{3}{2}}} \sum_{j=0}^A  \frac{k_j\Delta^j \psi_y( \lambda^{-1}\a)}{\lambda^j} +O_{A}\left(\frac{1}{|\lambda|^{\frac{5}{2}+A}}\right).
$$
Hence we conclude from \eqref{eq:neddy} that 
\begin{align*}
T(y)=
\frac{\ve \delta(\hat \c)}{
2\lambda^{\frac{3}{2}}} 
e\left(-\frac{\|\a\|^2}{2\lambda}\right)
\sum_{j=0}^A  \frac{k_j\Delta^j \psi_y( \lambda^{-1}\a)}{\lambda^j}
+O_{A}\left(\frac{\ve^3}{|\lambda|^{\frac{1}{2}}}+\frac{\ve}{ |\lambda|^{\frac{5}{2}+A}}\right).
\end{align*}
We now wish to 
substitute this into our  expression \eqref{eq:m4} for 
$I_r^*(\v)$.  In order to control the contribution from the error term, we apply \eqref{eq:seren} with $\ell=0$.
We 
therefore arrive at the statement of the lemma on 
redefining $k_j$ to be $k_j/2$.
\end{proof}

It remains to  consider the  integral 
\begin{equation}\label{eq:J}
\begin{split}
J_{j,q}(\c)
&=\int_\RR h\left(r,y\right) e\left(-\frac{\ve \hat v_4 y }{2}\right) \phi_j(y)\d y\\
&=
\int_\RR h\left(\frac{q}{Q},y\right) e\left(-\frac{\ve \hat c_4 y Q}{2q}\right) \phi_j(y)\d y,
\end{split}
\end{equation}
for $j\geq 0$.  
Recollecting \eqref{eq:psi}, all we shall need to know about $\phi_j$ is that it is a smooth compactly supported function with bounded derivatives, and that it does not depend on $q$. 
(Note that we may assume that $|(\hat c_1,\hat c_2, \hat c_3) |\ll \ve |\hat c_4|$ in what follows, since otherwise $\delta(\hat \c)=0$.)

\begin{lemma}\label{lem:partial}
Let $\c\in \mathcal{C}$
and  $k\in \{0,1\}$.
 Then
$$
q^k \frac{\partial^k J_{j,q,}(\c)}{\partial q^k} \ll_j  N^\delta.
$$
\end{lemma}

\begin{proof}
When $k=0$ the result follows immediately from \eqref{eq:seren}. 
Suppose next that $k=1$. Then  \eqref{eq:J} implies that 
\begin{align*}
 \frac{\partial J_{j,q}(\c)}{\partial q} =~& \frac{1}{Q}
 \int_\RR \frac{\partial h\left(r,y\right)}{\partial r} e\left(-\frac{\ve \hat c_4 y Q}{2q}\right) \phi_j(y)\d y\\ &+ 
 \int_\RR \frac{\pi i \ve \hat c_4	y Q}{q^2} h\left(r,y\right) e\left(-\frac{\ve \hat 
c_4 y Q}{2q}\right) \phi_j(y)\d y\\
=~&J_1+J_2,
\end{align*}
say. It follows from \eqref{eq:seren} that 
 $J_1\ll_j Q^{-1}r^{-1} =q^{-1}$, which is satisfactory. 
Next, a further application of \eqref{eq:seren} yields 
$$
J_2
\ll_j \frac{\ve |\hat c_4|Q}{q^2} 
 \int_\RR \left|y  h\left(r,y\right)\right|\d y
\ll_j \frac{\ve |\hat c_4|Q}{q^2}  \cdot r \leq \frac{N^\delta}{q},
 $$
for  $\c\in \mathcal{C}$.
\end{proof}

\section{Putting everything together}\label{s:final}

It is now time to return to \eqref{eq:2}, in order to conclude the proof of Theorem~\ref{thm1}. 

\subsection{The main term}

We begin by dealing with the main contribution, which comes from the term $\c=\0$.  Denoting this by $M(w)$, we see that 
\begin{equation}\label{eq:Mw}
M(w)=\frac{1}{Q^2}\sum_{q\ll Q} q^{-4}S_q(\0)I_q(\0)+O_A(Q^{-A}),
\end{equation}
for any $A>0$. 

In view of  \eqref{eq:psi}, 
$\psi_0(\x)$ is equal to 
$$ 
w_0\left(2\ve^{-2}(
-1+\sqrt{1-\ve^2\|\x\|^2})\right)\\
w_0\left(\sqrt{\|\x\|^2+
\ve^{-2}(1-\sqrt{1-\ve^2\|\x\|^2})^2}\right).
$$
As in 
\eqref{eq:approximate}, when
 $\psi_0(\x)\neq 0$ we must have 
\begin{align*}
2\ve^{-2}\left(-1+\sqrt{1-\ve^2\|\x\|^2}\right) &=-\|\x\|^2+O(\ve^2)\\
\|\x\|^2+
\ve^{-2}(1-\sqrt{1-\ve^2\|\x\|^2})^2&=\|\x\|^2+O(\ve^2). 
\end{align*}
In particular it is clear that 
\begin{equation}\label{eq:sigma}
\sigma_\infty=
\int_{\RR^3} \psi_0(\x)
\d\x\gg 1,
\end{equation}
for an absolute implied constant.
We now establish the following result.

\begin{lemma}\label{lem:I0}
We have  
$$
I_q(\0)
=\tfrac{1}{2}
\ve^5N^2 
\sigma_\infty
 \left(1+O(\ve^2)+O_A\left((q/Q)^A\right)\right),
$$
for any $A>0$, 
where  $\sigma_\infty$ is given by \eqref{eq:sigma}.
\end{lemma}

\begin{proof}
Returning to \eqref{eq:goaty}, it follows from 
\eqref{eq:m4} and \eqref{eq:m4'} that 
\begin{align*}
I_q(\0)
&=
\ve^4N^2 
 \int_\RR h(r,y) 
K(y) \d y,
\end{align*}
where 
$$
K(y)=\int_{\RR^3}
w_0(\|\u\|)
w_0( 2u_4/\ve)\, \frac{\d u_1\d u_2 \d u_3}{2/\ve +2u_4},
$$
and $u_4$ is given in terms of $y,u_1,u_2,u_3$ by \eqref{eq:u4}. 
Using  \eqref{eq:m4''}, we may write
$$
K(y)=\frac{\ve}{2}\left(1+O(\ve^2)\right) K^*(y), 
\quad \text{ with }
K^*(y)=\int_{\RR^3} \psi_y(\x)\d\x.
$$
The integral $K^*(y)$ is a smooth weight function  belonging to the class of weight functions considered in \cite[Lemma 9]{HB}. Noting from \eqref{eq:sigma} that $K^*(0)=\sigma_\infty$, 
it  therefore follows from this result that 
$$
 \int_\RR h(r,y) 
K^*(y) \d y=\sigma_\infty+O_A(r^A),
$$
for any $A>0$. We therefore deduce that 
\begin{align*}
I_q(\0)
=\tfrac{1}{2}
\ve^5N^2 
\sigma_\infty
 \left(1+O(\ve^2)+O_A(r^A)\right),
 \end{align*}
which completes the proof of the lemma.
\end{proof}

Now it is clear from  \S\ref{s:sum} 
that 
$
q^{-4}|S_q(\c)|\leq 4q^{-2}|S(m,n;q)|, 
$
for any vector $\c\in \ZZ^4$,  
where $(m,n)$ is $(N,F(\hat \c)/4)$, $(N/2,F(\hat \c)/2)$ 
or $(N/4,F(\hat \c))$ depending on whether $4\mid q$, $q \equiv 2 \bmod{4}$ or $2\nmid q$, respectively. 
Hence it follows from \eqref{eq:weil}, together with  the standard estimate for the divisor function, that 
\begin{equation}\label{eq:view}
\begin{split}
\sum_{t/2< q\leq t} q^{-4}|S_q(\c)|
\ll \sum_{t/2< q\leq t} q^{-2}
|S(m,n;q)|
&\ll_\delta t^{\delta/2} \sum_{t/2<q\leq t} \frac{\sqrt{(q,N)}}{q^{3/2}}\\
&\ll_\delta t^{-1/2+\delta/2} N^{\delta/2}, 
\end{split}
\end{equation}
for any $t>1$ and any  $\delta>0$.
Returning to \eqref{eq:Mw}, we may now conclude from Lemma \ref{lem:I0} 
and \eqref{eq:view} with $\c=\0$, 
that the contribution to $M(w)$ from $q\leq Q^{1-\delta}$ is 
\begin{align*}
&=\frac{1}{Q^2}\sum_{q\leq Q^{1-\delta}} q^{-4}S_q(\0)I_q(\0)+O_A(Q^{-A})\\
&=\frac{\ve^5N^2 }{2Q^2}
\sigma_\infty \mathfrak{S}(Q^{1-\delta}) +O\left(\frac{\ve^7N^{2+\delta/2}}{Q^2}\right) 
+O_A(Q^{-A}),
\end{align*}
where
$$
\mathfrak{S}(t)=\sum_{q\leq t} q^{-4}S_q(\0).
$$
This sum is absolutely convergent and satisfies 
$\mathfrak{S}(t)=\mathfrak{S}+O_\delta(t^{-1/2+\delta/2}N^{\delta/2})$,
for any $\delta>0$, 
by \eqref{eq:view}. Here, in the usual way,  
 $\mathfrak{S}$ is the Hardy--Littlewood product of local densities recorded in  \eqref{eq:density}.

Next, on invoking 
\eqref{eq:view}, once more, 
 the contribution from 
$q> Q^{1-\delta}$ is 
\begin{align*}
&\ll_A \frac{\ve^5N^2}{Q^2}\sum_{q> Q^{1-\delta}} q^{-4}|S_q(\0)|+Q^{-A}
\ll \frac{\ve^5N^{2+\delta/2}Q^{\delta/2}}{Q^{5/2}}. 
\end{align*}
Hence we have established the following result,  on recalling that $Q=\ve\sqrt{N}$,
which shows that the main term  is satisfactory for Theorem \ref{thm1}.

\begin{lemma}
For any $\delta>0$ we  have 
$$
M(w)=\frac{\ve^3N \sigma_\infty \mathfrak{S}}{2}
+O_\delta\left(
\ve^5N^{1+\delta}+
\ve^{\frac{5}{2}}N^{\frac{3}{4}+\delta}
\right).
$$
\end{lemma}

\subsection{The error term}

It remains to analyse the contribution $E(w)$, say, to $\Sigma(w)$ from vectors $\c\neq \0$
in \eqref{eq:2}. 
According to our work in 
\S\ref{s:sum} the value of $S_q(\c)$ differs according to the residue class of $q$ modulo $4$. We have 
$$E(w)=\sum_{i\bmod 4}E_i(w),
$$ 
where $E_i(w)$ denotes the contribution from $q\equiv i\bmod{4}$. 
Recall the definition of $\mathcal{C}$ from after the statement of Lemma \ref{largec}. 
In order to unify our  treatment of the four cases,  we write $\mathcal{C}_1=\mathcal{C}_2=\mathcal{C}$ and we denote by
$\mathcal{C}_2$ (resp.~$\mathcal{C}_4$)  the set of $\c\in \mathcal{C}$ for which $2\nmid c_1\dots c_4$ (resp.~$2\mid \c$). 
It will also be convenient to  set 
\begin{gather*}
(m_1,n_1)=
(m_3,n_3)=(N/4,F(\c)),\\
(m_2,n_2)=(N/2,F(\c)/2), \quad (m_4,n_4)=(N,F(\c)/4).
\end{gather*} 
In particular,  $m_in_i=NF(\hat \c)/4>0$ for $1\leq i\leq 4$, since $F(\c)=F(\hat \c)$.  

Let $1\ll R\ll Q$. We denote by $E_i(w,R)$ the overall contribution to $E_i(w)$ from $q\sim R$. (We write $q\sim R$ to denote $q\in (R/2,R]$.)
On recalling  \eqref{eq:goaty}, 
 it follows from our work so far that 
\begin{equation}\begin{split}
\label{eq:Ei}
E_i(w,R) 
&\ll \frac{1}{Q^2}
\sum_{\substack{\c \in \mathcal{C}_i\\ \c\neq \0}}
\left|\sum_{\substack{q\sim R\\ q\equiv i\bmod{4}}} q^{-2} S(m_i, n_i;q)
I_q(\c)\right|\\
&\ll  \frac{\ve^4N^2}{Q^2}
\sum_{\substack{\c\in \mathcal{C}_i\\ \c\neq \0}}
\left|\sum_{\substack{q\sim R\\
q\equiv i\bmod{4}
}} q^{-2}S(m_i,n_i;q)e_{r}(-\varepsilon^{-1}\c.\boldsymbol{\xi})I_r^{*}(\v)\right|.
\end{split}
\end{equation}

\subsubsection*{Contribution from large  $q$} 
Suppose first that $R\geq Q^{1-\eta}$, for some small $\eta>0$. (The choice $\eta=2\delta$ is satisfactory.)
We have 
$$
e_{r}(-\varepsilon^{-1}\c.\boldsymbol{\xi})=
e\left(\frac{2\sqrt{m_in_i}}{q}\alpha\right),
$$
with 
\begin{align*}
|\alpha|&=\varepsilon^{-1}|\hat c_4|
\cdot \frac{Q}{q} \cdot  \frac{q}{2\sqrt{m_in_i}}=
\frac{|\hat c_4| }{\sqrt{F(\hat \c)}} \leq 1.
\end{align*}
It now follows from Conjecture \ref{c:linnik}  that 
\begin{equation}\label{eq:L(t)}
L(t)=
\sum_{\substack{q\leq t\\ q\equiv i\bmod{4}}} \frac{S(m_i,n_i;q)}{q}
e\left(\frac{2\sqrt{m_in_i}}{q}\alpha\right)
\ll_\delta (tN)^\delta.
\end{equation}
Applying partial summation, based on Lemma \ref{lem:jumper}, 
we deduce that 
\begin{align*} 
E_i(w,R) 
&\ll_\delta  \frac{\ve^5N^{2+O(\delta)}}{Q^3}\cdot \frac{Q^2}{R^2}\cdot
\sum_{\substack{\c\in \mathcal{C}_i\\ \c\neq \0}}\frac{1}{\max\{1,\ve|\hat c_4|Q/R\}^{\frac{3}{2}}}\\
&\ll_\delta  \frac{\ve^5N^{2+O(\delta)}}{QR^2}\cdot \frac{\ve^{-1}R}{Q}\\
&=  \frac{\ve^4N^{2+O(\delta)}}{Q^2R}.
\end{align*}
Since $R\geq Q^{1-\eta}$, we deduce that 
$$
E_i(w,R) 
\ll_{\delta}
  \frac{\ve^4N^{2+O(\delta)}Q^\eta}{Q^3}
\leq
 \ve N^{\frac{1}{2}+O(\delta)+\eta}.
$$
This is satisfactory for Theorem \ref{thm1}, on redefining the choice of $\delta$, provided that  $\eta$ is small enough.

\subsubsection*{Contribution from small $q$ and small $\ve|\hat v_4|$} 

For the rest of the proof we suppose that $R<Q^{1-\eta}$. 
 Let us put $$
 \b=(\hat c_1,\hat c_2,\hat c_3),
 $$ 
 so that $\a=r^{-1}\b$ in Lemmas \ref{lem:I-small} and \ref{lem:I-stationary}.
Let $E_i^{(\text{small})}(w,R)$ denote the contribution to $E_i(w,R)$ from $\c\in \mathcal{C}_i$ such that 
\begin{equation}\label{eq:cat}
\ve|\hat c_4|\leq \frac{R^{1+\delta}}{Q}.
\end{equation}
In this case it is advantageous to apply Lemma \ref{lem:I-small} to evaluate $I_r^*(\v).$  
To begin with,  we consider the effect of substituting the main term 
from Lemma \ref{lem:I-small}.  
Noting that $(\ve \hat v_4)^{-1}\a=(\ve \hat c_4)^{-1}\b$ does not depend on $q$, 
we deduce from  \eqref{eq:Iy} that the only dependence on $q$ in $I(y)$ comes through the term 
$$
e\left(\frac{\ve \hat v_4}{2}\|\x\|^2-\a.\x\right)=e_r\left(\frac{\ve \hat c_4}{2}\|\x\|^2-\b.\x\right),
$$
in the integrand. 
Thus,  the main term 
in Lemma \ref{lem:I-small} makes 
the overall contribution
\begin{equation}\label{eq:anna}
\ll \frac{\ve^5N^2}{Q^2}
\sum_{\substack{\c\in \mathcal{C}_i\\ \c\neq \0\\ 
\text{\eqref{eq:cat} holds}}} 
\left|
\sum_{\substack{q\sim R \\ q\equiv i\bmod{4}}} \frac{S(m_i,n_i;q)}{q^2}e_{r}(-\varepsilon^{-1}\c.\boldsymbol{\xi})
I(0)
\right|
\end{equation}
to $E_i^{\text{(small)}}(w,R)$, where we recall from \eqref{eq:Iy}
that 
$$
I(0)=\int_{\RR^3} \psi_0(\x) e_r\left(\frac{\ve \hat c_4}{2}\|\x\|^2-\b.\x\right) \d\x.
$$
If $\c\neq \0$ and $|\hat c_4|\le \frac{1}{100}$ then 
$$
\|\b\|^2 =F(\hat \c)-\hat c_4^2=F(\c)-\hat c_4^2\gg 1.
$$ 
It therefore follows from \cite[Lemmas 3.1 and 3.2]{HBP} that 
$$
I(0)\ll_A \left(\frac{q}{|\b|Q}\right)^A \ll_A   Q^{-\eta A},
$$
since $q\leq Q^{1-\eta}$ in this case. The overall contribution to \eqref{eq:anna} from vectors $\c$  such that $ |\hat c_4|\le \frac{1}{100}$ is therefore seen to be satisfactory.

On interchanging the sum and the integral we are left with the contribution 
\begin{equation}\label{eq:anna+}
\ll \frac{\ve^5N^2}{Q^2}
\sum_{\substack{\c\in \mathcal{C}_i\\ |\hat c_4| > \frac{1}{100}\\ 
\text{\eqref{eq:cat} holds}}} 
\int_{[-1,1]^3}
|M_i(\x)|\d\x,
\end{equation}
where 
$$
M_i(\x)=
\sum_{\substack{q\sim R \\ q\equiv i\bmod{4}}} \frac{S(m_i,n_i;q)}{q^2}e_{r}(-\varepsilon^{-1}\c.\boldsymbol{\xi})
e_r\left(\frac{\ve \hat c_4}{2}\|\x\|^2-\b.\x\right).
$$
But 
$$
e_{r}(-\varepsilon^{-1}\c.\boldsymbol{\xi})
e_r\left(\frac{\ve \hat c_4}{2}\|\x\|^2-\b.\x\right)=
e\left(\frac{2\sqrt{m_in_i}}{q}\alpha\right),
$$
with 
\begin{align*}
\alpha&=\left(-\varepsilon^{-1}\hat c_4
+\frac{\ve \hat c_4 \|\x\|^2 }{2}-\b.\x 
\right)\cdot \frac{Q}{q} \cdot  \frac{q}{2\sqrt{m_in_i}}\\
&
=
-\frac{\hat c_4 }{\sqrt{F(\hat \c)}} 
+
\frac{\ve^2 \hat c_4 \|\x\|^2 }{2\sqrt{F(\hat \c)}}-
\frac{\ve \b.\x }{\sqrt{F(\hat \c)}}.
\end{align*}
But the inequality
$
\max\{\|\b\|,|\hat c_4|\}\leq \sqrt{F(\hat \c)},
$ 
implies that 
$|\alpha|\leq 1+O(\ve)$, since
$\x\in [-1,1]^3$.
Thus it follows from combining partial summation with Conjecture \ref{c:linnik}  that 
$
M_i(\x)\ll_\delta R^{-1}N^{\delta}.
$
 (Recall  that $\ve^{-1}\leq \sqrt{N}$ and $R\leq Q^{1-\eta}\leq Q$.)
Returning to \eqref{eq:anna+}, we conclude that the 
overall contribution to $E_i^{\text{(small)}}(w,R)$ from the 
main term in Lemma \ref{lem:I-small} is
\begin{align*}
\ll_\delta \frac{\ve^5N^{2+\delta}}{RQ^{2}}\#\left\{\c\in \mathcal{C}_i: 
\text{$|\hat c_4|> \tfrac{1}{100}$ and \eqref{eq:cat} holds}\right\}
&\ll_\delta \frac{\ve^4N^{2+4\delta}R^\delta}{Q^{3}}\\
&\ll_\delta \ve N^{\frac{1}{2}+5\delta}.
\end{align*}
This is satisfactory for Theorem \ref{thm1}.

It remains to study the effect of 
substituting the error term
from Lemma~\ref{lem:I-small} into 
\eqref{eq:Ei}.
Since $r\leq R/Q\leq Q^{-\eta}$ and 
$\ve|\hat v_4|=r^{-1}\ve|\hat c_4|\ll R^\delta$, by \eqref{eq:cat}, we see that the error term is 
\begin{align*}
\ll_A
\ve^3(1+\ve|\hat v_4|) +\ve(1+\ve|\hat v_4|)^{A}r^{A}
&\ll_A
\ve^3R^\delta +\ve R^{\delta A}Q^{-\eta A}\\
&\leq
\ve^3R^\delta +\ve Q^{A(\delta-\eta)}.
\end{align*}
On ensuring  that  $\delta<\eta$, we see that the second term
is an arbitrary negative power of $Q$ and so makes a satisfactory overall contribution to 
$E_i^{\text{(small)}}(w,R)$.
In view of   \eqref{eq:view}, the  
contribution from the term $\ve^3 N^\delta$ is found to be 
\begin{equation}\label{eq:shoe}
\begin{split}
\ll_\delta\frac{\ve^7 N^{2+\delta}}{Q^2R^{\frac{1}{2}}} 
\cdot \#\mathcal{C}_i
&\ll_\delta \frac{\ve^7 N^{2+\delta}}{Q^2} 
\cdot \ve^{-1}N^{4 \delta}
= \frac{\ve^6 N^{2+5\delta}}{Q^{2}} ,
\end{split}
\end{equation}
since $R\gg 1$.
The right hand side is $\ve^{4}N^{1+5\delta}$, which is also satisfactory for Theorem \ref{thm1}, on redefining $\delta$.

\subsubsection*{Contribution from small $q$ and large $\ve|\hat v_4|$} 

It remains to consider the case  $R<Q^{1-\eta}$ and
\begin{equation}\label{eq:dog}
\ve|\hat c_4|> \frac{R^{1+\delta}}{Q}.
\end{equation}
Let us write $E_i^{(\text{big})}(w,R)$ for the overall contribution to $E_i(w,R)$ from this final case.
Our main tool is now Lemma \ref{lem:I-stationary}.
Let $A\geq 0$. We begin by considering the effect of substituting the main term from this result into 
\eqref{eq:Ei}. 
This yields the contribution 
\begin{equation}\label{eq:anna'}
\ll  \frac{\ve^5N^2}{Q^2}
\sum_{\substack{\c\in \mathcal{C}_i\\ \text{\eqref{eq:dog} holds}}} \delta(\hat \c)
\sum_{j=0}^A \frac{|k_j|}{(\ve |\hat c_4|Q)^{\frac{3}{2}+j}}
|M_{i,j}|,
\end{equation}
where if $J_{j,q}(\c)$ is given by \eqref{eq:J}, then 
$$
M_{i,j}=
\sum_{\substack{q\sim R \\ q\equiv i\bmod{4}}} \frac{S(m_i,n_i;q)}{q}e_{r}(-\varepsilon^{-1}\c.\boldsymbol{\xi})
e_r\left(-\frac{\|\b\|^2}{2\ve \hat c_4}\right) q^{\frac{1}{2}+j} J_{j,q}(\c).
$$
Our plan is to use partial summation to remove the factor
$q^{\frac{1}{2}+j}J_{j,q}(\c)$.

First, as before, we note  that 
$$
e_{r}(-\varepsilon^{-1}\c.\boldsymbol{\xi})
e_r\left(-\frac{\|\b\|^2}{2\ve \hat c_4}\right)=
e\left(\frac{2\sqrt{m_in_i}}{q}\alpha\right),
$$
where
\begin{align*}
\alpha&=\left(-\varepsilon^{-1}\hat c_4
-\frac{ \|\b\|^2 }{2\ve \hat c_4}
\right)\cdot \frac{Q}{q} \cdot  \frac{q}{2\sqrt{m_in_i}}\\
&
=
-\left(\frac{\hat c_4 }{\sqrt{F(\hat \c)}} 
+
\frac{ \|\b\|^2 }{2\hat c_4 \sqrt{F(\hat \c)}}\right).
\end{align*}
We have  $|\alpha|\leq 1+O(\ve^2)$, since $\|\b\|\ll \ve |\hat c_4|$ when $\delta(\hat \c)\neq 0$.
Applying partial summation, based on \eqref{eq:L(t)} and Lemma \ref{lem:partial}, 
 we  deduce that 
$$
M_{i,j}=O_{j,\delta}(R^{\frac{1}{2}+j}N^{3\delta}).
$$
 Returning to \eqref{eq:anna'}, we 
conclude that the overall contribution 
to $E_i^{\text{(big)}}(w,R)$ 
from the main term in 
Lemma \ref{lem:I-stationary}
is 
\begin{align*}
\ll_{\delta,A}  \frac{\ve^5N^{2+3\delta}}{Q^2}\sum_{j=0}^A
\sum_{\substack{\c\in \mathcal{C}_i\\ \text{\eqref{eq:dog} holds}}} 
 \frac{R^{\frac{1}{2}+j}}{(\ve |\hat c_4|Q)^{\frac{3}{2}+j}}
&\ll_{\delta,A}  \frac{\ve^5N^{2+3\delta}}{Q^2}
\cdot \frac{N^{3\delta}}{\ve Q}
=\ve N^{\frac{1}{2}+6\delta}. 
\end{align*}
This is satisfactory for
Theorem \ref{thm1}, on redefining $\delta$.

We must now consider the effect of substituting the error term 
$$ 
\ll_{A}\frac{\ve^3}{|\ve\hat v_4|^{\frac{1}{2}}}+\frac{\ve}{ |\ve\hat v_4|^{\frac{5}{2}+A}}
$$
from Lemma~\ref{lem:I-stationary} into 
\eqref{eq:Ei}. Since $q\sim R$, it follows from 
\eqref{eq:dog} that 
$\ve |\hat v_4|\gg R^{\delta}$. 
The first term is therefore $O(\ve^3)$, which makes a satisfactory overall 
 contribution by  \eqref{eq:shoe}.
On the other hand, on invoking once more the argument in \eqref{eq:view}, the second term makes the overall contribution
\begin{align*}
&\ll_A  \frac{\ve^5N^2}{Q^2}
\sum_{\substack{\c\in \mathcal{C}_i\\ \text{\eqref{eq:dog} holds}}}
\sum_{\substack{q\sim R
}} \frac{q^{-2}|S(m_i,n_i;q)|}{|\ve \hat v_4|^{\frac{5}{2}+A}}\\
&\ll_{A,\delta}  \frac{\ve^5N^{2+\delta}}{R^{\frac{1}{2}}Q^2}
\left(\frac{R}{\ve Q}\right)^{\frac{5}{2}+A}
\sum_{\substack{\c\in \mathcal{C}_i\\ \text{\eqref{eq:dog} holds}}}
\frac{1}{|\hat c_4|^{\frac{5}{2}+A}}\\
&\ll_{A,\delta}  \frac{\ve^4 N^{2+4\delta}R^{\frac{1}{2}-A\delta}}{Q^3 }.
\end{align*}
This is $O_\delta(\ve N^{\frac{1}{2}+4\delta})$ on assuming that $A$ is 
is chosen so that $A\delta>\frac{1}{2}$.
This is also  satisfactory
 for Theorem \ref{thm1}, which thereby completes its proof.

\end{document}